\documentclass[11pt]{amsart}

\usepackage[margin=1in]{geometry}
\usepackage{enumerate, amsmath, amsthm, amsfonts, amssymb, mathrsfs, graphicx, paralist, ulem}
\usepackage[usenames, dvipsnames]{color}
\usepackage{breqn} 
\usepackage{tikz-cd}
\usepackage{hyperref}

\numberwithin{equation}{section}
\newtheorem{Theorem}[equation]{Theorem}
\newtheorem{Proposition}[equation]{Proposition}
\newtheorem{Lemma}[equation]{Lemma}

\newtheorem{Conjecture}[equation]{Conjecture}

\theoremstyle{definition}
\newtheorem{Remark}[equation]{Remark}

\newtheorem{eg}[equation]{Example}

\newcommand{\cA}{\mathcal{A}}

\newcommand{\bC}{\mathbb{C}}

\newcommand{\cG}{\mathcal{G}}

\newcommand{\cM}{\mathcal{M}}

\newcommand{\bN}{\mathbb{N}}
\newcommand{\cN}{\mathcal{N}}

\newcommand{\cO}{\mathcal{O}}

\newcommand{\cP}{\mathcal{P}}

\newcommand{\cW}{\mathcal{W}}

\newcommand{\cY}{\mathcal{X}}

\newcommand{\bc}{\mathbf{c}}

\newcommand{\fg}{\mathfrak{g}}

\newcommand{\bu}{\mathbf{u}}

\newcommand{\bv}{\mathbf{v}}

\newcommand{\CC}{\mathbb{C}}

\newcommand{\OO}{\mathbb{O}}

\newcommand{\ZZ}{\mathbb{Z}}

\newcommand{\Matnxn}{\rm{Mat}_{n}}

\renewcommand{\phi}{\varphi}

\newcommand{\ol}[1]{\overline{#1}}

\newcommand{\leftexp}[2]{\vphantom{#2}^{#1} #2}

\newcommand{\arxiv}[1]{\href{http://arxiv.org/abs/#1}{{\tt arXiv:#1}}}

\makeatletter
\def\Ddots{\mathinner{\mkern1mu\raise\p@
\vbox{\kern7\p@\hbox{.}}\mkern2mu
\raise4\p@\hbox{.}\mkern2mu\raise7\p@\hbox{.}\mkern1mu}}
\makeatother

\DeclareMathOperator{\im}{Im}

\renewcommand{\hom}{\operatorname{Hom}}

\DeclareMathOperator{\Spec}{Spec}

\newcommand{\Gr}{\mathbf{Gr}}

\newcommand{\fsl}{\mathfrak{sl}}

\newcommand{\suchthat}{\mid}

\renewcommand{\Gr}{\cG r}

\newcommand{\wt}{\text{wt}}

\newcommand{\bfnG}{\mathbf{G}}
\newcommand{\bfnN}{\mathbf{N}}

\begin{document}

\title{}
\author{Joel Kamnitzer}
\address{J.~Kamnitzer: Department of Mathematics, University of Toronto, Canada}
\email{jkamnitz@math.toronto.edu}

\author{Dinakar Muthiah}
\address{D.~Muthiah: Department of Mathematical and Statistical Sciences, University of Alberta, Canada}
\email{muthiah@ualberta.ca}

\author{Alex Weekes}
\address{A.~Weekes: Perimeter Institute for Theoretical Physics, Canada}
\email{aweekes@perimeterinstitute.ca}

\author{Oded Yacobi}
\address{O.~Yacobi: School of Mathematics and Statistics, University of Sydney, Australia}
\email{oyacobi@maths.usyd.edu.au}
\title{Reducedness of affine Grassmannian slices in type A}

\begin{abstract}
We prove in type A a conjecture which describes the ideal of transversal slices to spherical Schubert varieties in the affine Grassmannian.  As a corollary, we prove a modular description (due to Finkelberg-Mirkovic) of the spherical Schubert varieties.
\end{abstract}

\maketitle

\section{Introduction}
\subsection{The Reducedness Conjecture}

 Let $ G $ be a complex semisimple group and consider the associated  thick affine Grassmannian $\Gr=G((t^{-1}))/G[t]$.  There is a Poisson structure on $G((t^{-1}))$ arising from the Manin triple $(\fg((t^{-1})),\fg[t],t^{-1}\fg[[t^{-1}]])$.  Hence $G[t]$ is a Poisson subgroup of $G((t^{-1}))$, and this coinduces a Poisson structure on $\Gr$.

 For a dominant coweight $\lambda$ of $G$ consider the $G[t]$-orbit given by $\Gr^\lambda=G[t]t^\lambda$.  Note that the thin affine Grassmannian $G[t,t^{-1}]/G[t]$ is isomorphic to  the union $\bigcup \Gr^\lambda$ over all dominant coweights.

Given a pair of dominant coweights such that $\mu \leq \lambda$ we have that $\Gr^\mu \subset \overline{\Gr^\lambda}$.  Our main objects of interest are transversal slices to $\Gr^\mu$ inside $\ol{\Gr^\lambda}$, which we denote $\Gr_\mu^\lambda$.  These slices arise in several contexts:
\begin{enumerate}
\item By the Geometric Satake Correspondence, the intersection homology of   $\Gr_\mu^\lambda$ can be identified with the associated graded of $V(\lambda)_\mu$, the $\mu$ weight space of the irreducible $^LG$ module of highest weight $\lambda$ \cite{L},\cite{G},\cite{MV}.
\item $\Gr_\mu^\lambda$ is a conical Poisson subvariety of the affine Grassmannian with respect to loop rotation \cite{KWWY}.

\item By recent work of Braverman, Finkelberg, and Nakajima \cite{BFN2}, these slices are the Coulomb branches of certain $3d$ $\cN=4$ quiver qauge theories.
\item Closely related to the previous item, the slices in the affine Grassmannian are conjectured to be symplectic dual to corresponding Nakajima quiver varieties \cite{BLPW}.
\end{enumerate}

The transversal slice is constructed as an intersection $\Gr_\mu^\lambda=\Gr_\mu \cap \overline{\Gr^\lambda}$, where $\Gr_\mu$ is an orbit of the congruence subgroup of $G[[t^{-1}]]$ acting on $\Gr$.  The Reducedness Conjecture describes the ideal of $\Gr_\mu^\lambda$ inside $\Gr_\mu$.  More precisely, in \cite{KWWY} a Poisson ideal $J_\mu^\lambda \subset \cO(\Gr_\mu)$ is explicitly defined via generators and it is shown that the vanishing of this ideal is  $\Gr_\mu^\lambda$.  Let $\cY_\mu^\lambda$ be the (possibly non-reduced) scheme whose ideal is $J_\mu^\lambda$.

\begin{Conjecture}[Conjecture 2.20, \cite{KWWY}]
\label{MainConj}
The ideal of $\Gr_\mu^\lambda$ is $J_\mu^\lambda$.  Equivalently, $\cY_\mu^\lambda$ is reduced.
\end{Conjecture}

Our aim is to prove this conjecture in type A:
\begin{Theorem}
\label{MainThm}
Let $G=SL_n$.  Then Conjecture \ref{MainConj} holds.
\end{Theorem}

For $G=SL_2,SL_3$, along with some special cases for general $n$, this conjecture is  proved in \cite{KMW}.  We will show how the main result of \cite{KMW} along with two additional ingredients proves Theorem \ref{MainThm}.  These ingredients are a) Weyman's description of the ideals defining nilpotent orbits in $\fsl_n$, and b) an isomorphism  motivated by \cite{BFN2} between certain $\cY_\mu^\lambda$ for different $n$.

\subsection{Consequences of the Reducedness Conjecture}

Before describing the proof of Theorem \ref{MainThm}, we discuss some implications of Conjecture \ref{MainConj}.

\subsubsection{Truncated shifted Yangians}
\label{intro: truncated shifted Yangians}
The main aim of \cite{KWWY} is to introduce a quantum group, the so-called truncated shifted Yangian, in order to quantize the slice $\Gr_\mu^\lambda$.

The truncated shifted Yangian is defined in several steps.  Firstly, one constructs the shifted Yangian $Y_\mu$, a $\bC[\hbar]$-algebra which quantizes $\Gr_\mu$ in the sense that
$$
Y_\mu/\hbar Y_\mu \cong \bC[\Gr_\mu].
$$
Next, one defines the GKLO representation $\Psi_\mu^\lambda: Y_\mu \to D_\mu^\lambda$ which depends also on the parameter $\lambda$.  The target space $D_\mu^\lambda$ is an algebra of difference operators.  

Finally, one defines the truncated shifted Yangian to be image $\im(\Psi_\mu^\lambda)$ of the shifted Yangian under the GKLO representation.  

In \cite{KWWY} it is shown that the truncated shifted Yangian quantizes a scheme supported on $\Gr_\mu^\lambda$.  Furthermore it is shown that Conjecture \ref{MainConj} implies that this scheme is actually $\Gr_\mu^\lambda$.  Recently, in \cite[Appendix B]{BFN1}, the latter statement was proven for all simply-laced $G$ by identifying the truncated shifted Yangian with the quantized Coulomb branch $\cA_\hbar$ of a $3d$ $\cN=4$ quiver gauge theory.  

By \cite[Theorem 4.10]{KWWY}, Conjecture \ref{MainConj} also implies that one can give a set of explicit generators for the kernel of $\Psi_\mu^\lambda$.  We thereby obtain a presentation for the truncated shifted Yangian, or equivalently for the quantized Coulomb branch. Denoting this explicitly presented algebra by $Y_\mu^\lambda$, in general there is a diagram
\begin{equation} \label{yangians diagram}
\begin{tikzcd}
Y_\mu \drar[two heads] \dar[two heads] &  \\ Y_\mu^\lambda \rar[two heads] & \im(\Psi_\mu^\lambda)  \cong \cA_\hbar
\end{tikzcd}
\end{equation}
Theorem \ref{MainThm} implies that, in type A, the bottom row consists of isomorphisms.  

We remark that in \cite{KTWWY} the highest weight theory of the truncated shifted Yangian is studied via the algebra $Y_\mu^\lambda$.  Therefore Theorem \ref{MainThm} implies that the results in \cite{KTWWY} actually also hold for the algebra $\im(\Psi_\mu^\lambda)$.

\subsubsection{Modular description of spherical Schubert varieties}

Another consequence of Conjecture \ref{MainConj} is a modular description of the spherical Schubert varieties.  We give a brief indication of this connection; for a thorough discussion  see Sections 1 and 2 of \cite{KMW} (cf. also \cite[Remark 2.1.7]{Z}).

The modular description of the thin affine Grassmannian (due to Beauville-Laszlo) is as follows: Let $X$ be a smooth curve, and $x\in X$ a closed point.  Then an $S$--point of $\Gr$ consists of a pair $(\cP,\phi)$, where $\cP$ is a principal $G$--bundle on $S\times X$, and $\phi:\cP_0|_{S\times (X\setminus x)} \to \cP|_{S\times (X\setminus x)}$ is an isomorphism where $\cP_0$ is the trivial bundle on $S\times X$.

Finkelberg and Mirkovi\'c propose a modular description of the spherical Schubert varieties $\ol{\Gr^\lambda}$ \cite[Proposition 10.2]{FM}.  They consider pairs $(\cP,\phi)$ as above, where the pole of $\phi$ at $x$ is controlled by $\lambda$ (see \cite[Section 2]{KMW} for a precise description).  While  this description is correct set-theoretically, it is not clear that the moduli space is reduced.

Conjecture \ref{MainConj} together with \cite[Proposition 5.1]{KMW} implies this modular description of $ \ol{\Gr^\lambda} $ is correct for $ G = SL_n $.

%%%%%%%%%%%%%%%%%%%%%%%%%%%%%%
\subsection*{Acknowledgements}
%%%%%%%%%%%%%%%%%%%%%%%%%%%%%%
We thank Alexander Braverman, Michael Finkelberg, and Anthony Henderson for helpful conversations.  
J.K. is supported by NSERC and a Sloan fellowship. D.M. is supported by a PIMS Postdoctoral Fellowship. A.W. was supported in part by a Government of Ontario graduate scholarship.  O.Y. is supported by the Australian Research Council.   
This research was supported in part by Perimeter Institute for Theoretical Physics. Research at Perimeter Institute is supported by the Government of Canada through Industry Canada and by the Province of Ontario through the Ministry of Economic Development and Innovation.

%%%%%%%%%%%%%%%%%%%%%%%%%%%%%%%%%%%%%%%%%%%%%%%%%%%%%%%%%%%%%%
%%%%%%%%%%%%%%%%%%%%%%%%%%%%%%%%%%%%%%%%%%%%%%%%%%%%%%%%%%%%%%

\section{Definitions and Overview}
\label{SecNot}
\subsection{Definitions} We recall some notations and results from \cite{KMW,KWWY}.  We work throughout over $\bC$.

Let $G$ be a semisimple group, $\fg$ its Lie algbera, and let $I$ denote the nodes of the Dynkin diagram of $\fg$.
We write $\lambda,\mu,$ etc. for coweights of $G$ and $\lambda^\vee,\mu^\vee,$ etc. for weights of $G$.  Let $\varpi_i,\varpi_i^\vee$ (respectively $\alpha_i,\alpha_i^\vee$) be the fundamental coweights and fundamental weights of $\fg$ (resp. simple coroots and simple roots of $\fg$).  Let $w_0$ be the longest element of the Weyl group and set $\lambda^*=-w_0\lambda, \varpi_{i^*}=\varpi_i^*, \alpha_{i^*}=\alpha_i^*$. Let $(\cdot,\cdot)$ be the Killing form $\fg$, and for coweights $\mu,\lambda$ write $\mu \leq \lambda$ if $\lambda-\mu \in \bigoplus_i \bN\alpha_i$.
\begin{Remark}
In Sections \ref{sec:iso} and \ref{sec:bfn} we will work with two semisimple groups simultaneously.  When writing $w_0, \lambda^\ast$, etc., the relevant group will be clear from context.
\end{Remark}

Let $G((t^{-1}))$ (respectively $G[t],G[[t^{-1}]]$) be the $\bC((t^{-1}))$ points of $G$ (respectively the $\bC[t],\bC[[t^{-1}]]$ points of $G$).  Define also $G_1[[t^{-1}]] \subset G[[t^{-1}]]$ as the kernel of the evaluation $G[[t^{-1}]] \to G, t \mapsto \infty$.  A coweight $\lambda$ of $G$ can be considered as a $\bC((t^{-1}))$ point of $G$, and we let $t^\lambda$ denote its image in $\Gr$. There is a corresponding orbit $\Gr^\lambda = G[t] t^\lambda$, and spherical Schubert variety $\ol{\Gr^\lambda}$. Recall that $\Gr^\mu \subset \ol{\Gr^\lambda}$ if and only if $\mu \leq \lambda$.

Consider a pair of dominant coweights $\mu,\lambda$ such that $\mu \leq \lambda$.  Define $\Gr_\mu^\lambda=\Gr_\mu \cap \ol{\Gr^\lambda}$, where
$$
\Gr_\mu=G_1[[t^{-1}]]t^{w_0\mu}.
$$
$\Gr_\mu^\lambda$ is a transversal slice to $\Gr_\mu$ in $\ol{\Gr^\lambda}$ at the point $t^{w_0\mu}$.

 Let $V$ be a representation of $G$ and $v\in V,\beta\in V^*$.  We'll introduce functions $\Delta_{\beta,v}^{(s)}$ on  $G_1[[t^{-1}]]$.   Let $\Delta_{\beta,v} \in \cO(G)$ be the matrix coefficient: $g\mapsto \beta(gv)$.  Then $G_1[[t^{-1}]]$ acts on $V[[t^{-1}]]$ and for $g\in G_1[[t^{-1}]]$ we have
$$
\Delta_{\beta,v}(g)=\sum_{s\geq0}\Delta_{\beta,v}^{(s)}(g)t^{-s}.
$$

For instance, consider $G=SL_n$ and the representation $\bigwedge^i\bC^n$ (this case is sufficient for our purposes).  If we take $v,\beta$ to be standard basis and dual basis elements, then $\Delta_{\beta,v}(g)$ is a $i\times i$ minor of $g$, and $\Delta_{\beta,v}^{(s)}(g)$ extracts its $t^{-s}$ coefficient.  It will be convenient for us to sometimes use the notation $\Delta_{C,D}^{(s)}$ to denote this function, where the matrix minor is taken with respect to rows specified by $C$ and columns specified by $D$ (here $C,D \subset \{1,...,n\}$ both have cardinality $i$).

We refer to \cite[Section 2]{KWWY} for results concerning the Poisson structure on $\Gr$.  We recall that $G_1[[t^{-1}]]$ is a Poisson algebraic group, $\Gr_\mu$ is a Poisson homogenous space, and $\Gr_\mu^\lambda$ is a Poisson subvariety of $\Gr$.
The Poisson bracket on $\cO(G_1[[t^{-1}]])$ is specified by the following equations:
$$
\{\Delta_{\beta_1,v_1}^{(r+1)},\Delta_{\beta_2,v_2}^{(s)} \}-\{\Delta_{\beta_1,v_1}^{(r)},\Delta_{\beta_2,v_2}^{(s+1)} \}=\sum_a \Big( \Delta_{J_a\beta_1,v_1}^{(r)}\Delta_{J^a\beta_2,v_2}^{(s)}- \Delta_{\beta_1,J_av_1}^{(r)}\Delta_{\beta_2,J^av_2}^{(s)} \Big)
$$
for all $r,s\geq0$, where $\{J_a\},\{J^a\}$ are dual bases of $\fg$ with respect to $(\cdot,\cdot)$ (see \cite[Section 2.6]{KWWY}).

Let $V(\varpi_i^\vee)$ be the irreducible representation of $G$ of highest weight $\varpi_i^\vee$.  Fix a highest weight vector $v_i\in V(\varpi_i^\vee)$ and a lowest weight dual vector $v_i^*\in V(\varpi_i^\vee)^*$.  Write $\lambda-\mu=\sum_im_i\alpha_{i^*}$ (recall that $\mu\leq\lambda$, so all $m_i \geq 0$).
$\cO(\Gr_\mu)$ can be canonically identified with the functions on $G_1[[t^{-1}]]$ invariant under the stabilizer in $G_1[[t^{-1}]]$ of $t^{w_0\mu}$.  This is a Poisson subalgebra of $\cO(G_1[[t^{-1}]])$.  It can be shown that the functions $\Delta_{v_i^*,v_i}^{(s)}$ for $i\in I$ and $s>m_i$ are invariant under this stabilizer (\cite[Lemma 2.19]{KWWY}), and hence can be considered as functions on $\Gr_\mu$.

Now we can define the central objects appearing in Conjecture \ref{MainConj}:
\begin{itemize}
\item let $J_\mu^\lambda \subset \cO(\Gr_\mu)$ be the Poisson ideal generated by $\Delta_{v_i^*,v_i}^{(s)}$ for $i\in I$ and $s>m_i$, and
\item let $\cY_\mu^\lambda={^G}\cY_\mu^\lambda$ be the corresponding subscheme of $\Gr$.
\end{itemize}
Note that in \cite{KMW} $\cY_\mu^\lambda$ is defined as the intersection $\ol{\cY^\lambda}\cap \Gr_\mu$, where $\ol{\cY^\lambda}$ is given by the modular description of the orbit closures due to Finkelberg-Mirkovi\'c.  \cite[Theorem 8.4]{KMW} proves that the ideal of $\ol{\cY^\lambda}\cap \Gr_\mu$ is $J_\mu^\lambda$, and so here we take this as the definition.

\subsection{Overview of Proof}
\label{section: proof overview}
For the remainder of the paper, unless otherwise noted we let $G=SL_n$.
The proof of Theorem  \ref{MainThm} relies on the following three results.  \begin{Proposition}
\label{PropWeyman}
Suppose that $\lambda \leq n\varpi_1$.  Then $\cY^\lambda_0$ is reduced.
\end{Proposition}
Next  let $\lambda$ be an arbitrary coweight of $G$, such that $\lambda \geq \mu = 0$, and write it as a sum of simple coroots and also of fundamental coweights: $$\lambda=\sum m_j\alpha_{j^*}=\sum\lambda_j\varpi_{j^*}.$$  Set $k:=m_1 = \langle \lambda, \varpi_{n-1}^\vee \rangle $, and let $\tau$ be the standard inclusion of Dynkin diagrams $A_{n-1} \to A_{kn-1}$.  Denote by $\tau$ also the map on coweights, which extends $\varpi_i \mapsto \varpi_{\tau(i)}$ by linearity; for our chosen $\tau$ this is simply $\varpi_i \mapsto \varpi_i$.
\begin{Proposition}
\label{PropCalc}
There is an isomorphism
\begin{equation}
\label{eq:bfniso}
^{SL_n}\cY^\lambda_0 \cong {^{SL_{kn}}\cY^{\tau(\lambda)}_{k\varpi_n}}.
\end{equation}
\end{Proposition}
\begin{Proposition}[\mbox{\cite[Theorem 1.6]{KMW}}]
\label{PropKMW}
Let $\lambda$ be a dominant coweight.  If $\cY^\lambda_0$ is reduced then $\cY_\mu^\lambda$ is reduced for all $\mu \leq \lambda$.
\end{Proposition}

Given these results, and  a simple computation showing that $\tau(\lambda) \leq kn\varpi_1$ (Lemma \ref{lem:simpcomp}(b)), the proof of Theorem \ref{MainThm} is immediate.

It remains then only to prove the first two propositions.  Proposition \ref{PropWeyman} will be proved in Section \ref{SectionWeyman}, and Proposition \ref{PropCalc} in Section \ref{sec:iso}.  In Section \ref{sec:bfn} we   discuss how (\ref{eq:bfniso}) is motivated by as isomorphism between affine Grassmannian slices that follows from  \cite{BFN2}.  In Section \ref{sec:truncatediso} we prove that this isomorphism can be quantized using truncated shifted Yangians.

%%%%%%%%%%%%%%%%%%%%%%%%%%%%%%%%%%%%%%%%%%%%%%%%%%%%%%%%%%%%%%
%%%%%%%%%%%%%%%%%%%%%%%%%%%%%%%%%%%%%%%%%%%%%%%%%%%%%%%%%%%%%%

\section{Weyman's equations}
\label{SectionWeyman}

\subsection{The nilpotent cone}

Since $\Gr_0\cong G_1[[t^{-1}]]$ we can view $\cY^{n\varpi_1}_0$ as a subscheme of $G_1[[t^{-1}]]$.  Let $\lambda\geq 0$ be a dominant coweight of $G$ and as above set $\lambda=\sum m_i\alpha_{i^*}$.  Recall the following (which holds for arbitrary semisimple groups):
\begin{Proposition}[\mbox{\cite[Proposition 2.15]{KWWY}}]
\label{PropKWWY}
$J_0^\lambda$ is generated as an ordinary ideal  by $\Delta_{\beta,v}^{(s)}$ for $s>m_i$, where $i\in I$ and $\beta,v$ range over bases for $V(\varpi_i)^*$ and $V(\varpi_i)$.
\end{Proposition}
In the case when $\lambda=n\varpi_1$, then $m_i=i$, and we have in particular that $J_0^{n\varpi_1}$ contains $\Delta_{\beta,v}^{(s)}$, where $s>1$ and $\beta,v$ range over bases for $V(\varpi_1^\vee)^*$ and $V(\varpi_1^\vee)$.  It's easy to see that these elements are sufficient to generate the whole ideal.
Therefore we have
$$
\cY^{n\varpi_1}_0 \subset I_n +t^{-1}\Matnxn,
$$
where $\Matnxn$ denotes the affine space of $n\times n$ matrices.  We use this to define an embedding $\cY^{n\varpi_1}_0 \hookrightarrow \Matnxn$, by $I_n+t^{-1}X \mapsto X$.  The condition $\det(I_n+t^{-1}X)=1$ implies that the image of this map is precisely $\cN$, the nilpotent cone of $\fg$.  Therefore we have an isomorphism of schemes $\cY^{n\varpi_1}_0 \cong \cN$.  In particular, the reducedness conjecture is true and we have $ \cY^{n\varpi_1}_0 = \Gr^{n\varpi_1}_0 $.

\subsection{Nilpotent orbit closure}
Let $\lambda$ be a dominant coweight such that $0 < \lambda \leq n \varpi_1$. Recall that we have the expansions $\lambda = \sum \lambda_j \varpi_{j^*}= \sum m_i \alpha_{i^*}
$. We form the following partition, written in exponential notation:
\begin{align}
  \label{eq:4}
  \bv = 1^{\lambda_{n-1}} 2^{\lambda_{n-2}} \cdots (n-1)^{\lambda_{1}}
\end{align}
The condition $0 < \lambda \leq n \varpi_1$ implies that $\sum j \lambda_{n-j} = n$, i.e., that $\bv$ is a partition of $n$.

Let $\bu = \bv^T$ be the conjugate partition. Because $\lambda$ lies in the coroot lattice, we can also expand:

\begin{Lemma}
Let us write $\bu = (\bu_1 \geq \bu_2 \geq \cdots \geq \bu_{n-1})$. Then:
\begin{align}
  \label{eq:6}
  m_{n-i} = \bu_1 + \cdots + \bu_i - i
\end{align}
\end{Lemma}

Using the partition $ \bu $, let us consider the nilpotent orbit closure $\overline{\OO_\bu} \subseteq \cN$, where $ \OO_\bu $ denotes the orbit of nilpotent matrices whose Jordan form is given by $ \bu$.
\begin{Proposition} \label{prop:reduced-slice-and-nilpotent-orbit-closure}
The isomorphism $ \Gr^{n\varpi_1}_0 = \cY^{n\varpi_1}_0 \cong \cN $ identifies $ \Gr^\lambda_0 $ with $\overline{\OO_\bu} $.
\end{Proposition}

Using results of Weyman \cite{W}, we will now see that under the identification $ \cY^{n\varpi_1}_0 \cong \cN$, the subscheme $\cY^\lambda_0$ is equal to nilpotent orbit closure $ \overline{\OO_\bu} $ with its {\it reduced} induced scheme structure.

\subsection{Appearance of Weyman's equations}
Under the isomorphism $ \cY^{n\varpi_1}_0 \cong \cN $, we can identify $\cY^\lambda_0$ as a subscheme of the nilpotent cone.  More precisely, $\cY_0^\lambda$ is identified with the subscheme of $\cN$ given by the functions
$$
\{f_{C,D}^{(s)}\mid 1\leq k<n, |C|=|D|=k, \text{ and } s>m_k\},
$$
where $f_{C,D}^{(s)}(X)=\Delta_{C,D}^{(s)}(I_n+t^{-1}X)$.  As functions on $ \cN $, $ f_{C,D}^{(s)} $ is a sum of $s \times s$ minors.

Note that the ideal of $ \cN $ as a subscheme of $ \Matnxn $ is generated by (for each $ p $) the sum of all principal $ p \times p $ minors.

Let $W_{k,p} = \text{span} \{ f_{C,D}^{(p)} \suchthat |C| = |D| = k \}$, and let $U_{0,p}$ be the one-dimensional space spanned by the sum of all principal $p\times p$ minors.  We  then have that the ideal of $\cY^\lambda_0$ in $\cO(\Matnxn)$ is generated by
\begin{equation}
\label{eq:Yideal}
\bigoplus_{k,p>m_k} W_{k,p} \oplus \bigoplus_{p=1}^n U_{0,p}.
\end{equation}

Let $M_p \subset \cO(\Matnxn)$ be span of all $p\times p$ minors, and set $E=\bC^n$.  We let $GL_n$ act on $\cO(\Matnxn)$ by $(g\cdot f)(A)=f(g^{-1}Ag)$.  Under this action $M_p \cong \Lambda^p E \otimes \Lambda^p E^*$.  By the Pieri rule, we have:
\begin{align} \label{eq:10}
  M_p = \bigoplus_{0\leq k \leq \min{(p,n-p)}} U_{k,p}
\end{align}
where $U_{k,p}\cong S_{\alpha_k} E$, the Schur module of highest weight $\alpha_k=(1^k,0^{n-2k},(-1)^{k})$.
Note that in the case where $k=0$ this notation agrees with the  definition of $U_{0,p}$ in (\ref{eq:Yideal}).

It is clear that $W_{k,p}$ is a subrepresentation of $\cO(\Matnxn)$. Moreover, each $f_{C,D}^{(p)}$ is a weight vector for the torus of diagonal matrices in $GL_n$ with weight:
\begin{align}
  \label{eq:13}
  \wt f_{C,D}^{(p)} = -\sum_{i \in C} \varepsilon_i + \sum_{j \in D} \varepsilon_j
\end{align}
Suppose now that $0\leq k\leq \min{(p,n-p)}$, and set $C_0=\{k+1,...,n\}$ and $D_0=\{1,...,n-k\}$.  Then we have
$
   \wt f_{C_0,D_0}^{(p)} = \alpha_k.
$
 Furthermore $f_{C_0,D_0}^{(p)}$ is fixed by the unipotent upper triangular matrices in $GL_n$. Therefore, it generates a copy of $S_{\alpha_k}E$. Since $W_{n-k,p} \subset M_p$, by \eqref{eq:10} we see that
\begin{equation}
\label{eq:inclusion}
U_{k,p} \subset W_{n-k,p}.
\end{equation}

 We recall Weyman's Theorem on nilpotent orbit closures, which by the above lemma can be stated as follows:

\begin{Theorem}[\mbox{\cite[Theorem 4.6]{W}}]
The ideal of  $\overline{\OO_\bu}$ is generated by the following:
\begin{align}
  \label{eq:11}
   \bigoplus_{1 \leq i \leq n; p > m_{n-i}}  U_{i,p} \oplus \bigoplus_{1 \leq p \leq n} U_{0,p}
\end{align}
\end{Theorem}

By Equations (\ref{eq:Yideal}) and (\ref{eq:inclusion}), we see that Weyman's Theorem implies that $ I(\overline{\OO_\bu}) \subset I(\cY^{\lambda}_0) $. By \cite[Corollary 2.16]{KWWY} $V(J^{\lambda}_0)=\Gr_0^\lambda $ and $ \Gr_0^\lambda = \overline{\OO_\bu}$ by Proposition \ref{prop:reduced-slice-and-nilpotent-orbit-closure}.  Therefore these ideals have the same radical, and since $I(\overline{\OO_\bu})$ is radical we obtain the following isomorphism of subschemes:
\begin{Proposition}{\label{prop:isomorphism-of-slice-to-nilpotent-orbit-closure}}
  Let $\lambda$ be a dominant coweight for $SL_n$ with $\lambda \leq n \varpi_1$, and let $\bu$ be the corresponding partition of $n$. Then we have
\begin{align}
  \label{eq:01}
  \cY^{\lambda}_0 = \overline{\OO_\bu}
\end{align}
as subschemes of $n \times n$ matrices.
\end{Proposition}
 This shows that $\cY^{\lambda}_0$ is reduced for any $\lambda<n\varpi_1$, proving Proposition \ref{PropWeyman}.

%%%%%%%%%%%%%%%%%%%%%%%%%%%%%%%%%%%%%%%%%%%%%%%%%%%%%%%%%%%%%%
%%%%%%%%%%%%%%%%%%%%%%%%%%%%%%%%%%%%%%%%%%%%%%%%%%%%%%%%%%%%%%

\section{An isomorphism of slices}
\label{sec:iso}
\subsection{Some varieties of interest}
For the moment consider  the general setting where $G$ is a semisimple group, and $\lambda, \mu$ are dominant coweights for $G$ with $\lambda \geq \mu$.  In this case we denote $\lambda_i  = \langle \lambda^\ast, \alpha_i^\vee\rangle$, $\mu_i = \langle \mu^\ast, \alpha_i^\vee\rangle$ and $m_i = \langle \lambda^\ast-\mu^\ast, \varpi_i^\vee\rangle$.

As in \cite[Appendix B]{BFN2}, consider the subgroup
\begin{equation}
\label{eq: Gmu}
G_\mu := \left\{ g\in G_1[[t^{-1}]] : t^{-w_0\mu} g t^{w_0 \mu} \in G_1[[t^{-1}]] \right\},
\end{equation}
which has the property that there is an isomorphism $G_\mu \stackrel{\sim}{\longrightarrow} \Gr_\mu$ defined by $g\mapsto g t^{w_0 \mu}$.  Hence the translated subset \begin{equation}
\label{eq: cosets for Grmu}
\cW_\mu := G_\mu \cdot t^{w_0\mu} \subset G((t^{-1}))
\end{equation}
maps bijectively onto $\Gr_\mu$ under the quotient map $G((t^{-1})) \rightarrow \Gr_G$.  It is naturally an affine scheme of infinite type.  

Consider a closed subscheme $\cW_\mu^\lambda$ of $\cW_\mu$, which is defined by imposing the following conditions on $g t^{w_0\mu} \in \cW_\mu$:
\begin{quote}
For every dominant weight $\tau^\vee$ of $G$, the valuation of $g t^{w_0 \mu}$ acting on $V(\tau^\vee) (( t^{-1}))$ is greater than or equal to $\langle \lambda,  w_0 \tau^\vee \rangle $.
\end{quote}
The above can be understood in terms of matrix coefficients, like in Section \ref{SecNot}. Note that in \cite{KMW} $\overline{\cY^\lambda}$ is defined by these same conditions on valuations, but applied to a coset representative $[ h]\in \Gr$ (i.e. with $h$ in place of $g t^{w_0 \mu}$ above).  Since $\cW_\mu$ provides a choice of coset representatives for $\Gr_\mu$, this implies:
\begin{Proposition}
\label{prop: slice via coset reps}
The quotient map $G((t^{-1})) \rightarrow \Gr_G$ induces an isomorphism of schemes $\cW_\mu^\lambda \cong \cY_\mu^\lambda$.
\end{Proposition}
\begin{proof}
This follows since $\cW_\mu^\lambda$ is the fibre product:
\[
\begin{tikzcd}
\cW_\mu^\lambda \rar \dar & \overline{\cY^\lambda} \dar[hook] \\
\cW_\mu \rar[hook] & \Gr_G
\end{tikzcd}
\]
where the bottom arrow comes from the quotient map $G((t^{-1})) \rightarrow \Gr_G$.

In other words, the definition of $\cW_\mu^\lambda$ is exactly a translation of the conditions defining the scheme-theoretic intersection $\Gr_\mu \cap \overline{\cY^\lambda}$ under the isomorphism $\cW_\mu \rightarrow \Gr_\mu$.
\end{proof}

\subsection{Explicit description of certain slices}

We return to the case where $G$ is a special linear group.  Fix a dominant $SL_n$-coweight $\lambda\geq 0$, and write $ \lambda = \sum \lambda_i \varpi_{i^\ast} = \sum m_i \alpha_{i^\ast} $.  Set $k = m_1$.  Then $\lambda \leq k n \varpi_1$, and this is the minimal value of $k$ such that this inequality holds.

As in Section \ref{section: proof overview}, we consider $SL_n \subset SL_{kn}$ corresponding to the inclusion $\tau$ of Dynkin diagrams $\{ 1,\ldots, n-1\} \subset \{1,\ldots, kn-1\}$. Define a map $\tau$ taking $SL_n$ coweights to $SL_{kn}$ coweights, which extends $\varpi_i \mapsto \varpi_i$ by linearity.

Our goal is to explicitly describe an isomorphism:
\begin{equation}
\label{eq: goal iso}
^{SL_n} \cY_0^\lambda \ \ \cong \ \ ^{SL_{kn}} \cY^{\tau(\lambda)}_{k \varpi_n}
\end{equation}
Using Proposition \ref{prop: slice via coset reps}, we will work exclusively with $\cW_\mu^\lambda$.  For simplicity, we will abuse notation and write continue to write $\cY_\mu^\lambda$.

\subsubsection{Case of $SL_n$} Since $^{SL_n} \cW_0 = (SL_n)_1[[t^{-1}]]$ and $\langle \lambda, w_0 \varpi_i^\vee \rangle = - m_i $, by Proposition \ref{prop: slice via coset reps} we have that
\begin{equation}
\label{eq: n presentation}
^{SL_n} \cY_0^\lambda = \left\{ g\in (SL_n)_1[[t^{-1}]] : \begin{matrix} \text{the valuation of any } i\times i \text{ minor}  \\ \text{of } g \text{ is } \geq - m_i\end{matrix} \right\}
\end{equation}
Indeed, it suffices to consider only $\tau^\vee = \varpi_i^\vee$ the fundamental weights for $SL_n$.

\subsubsection{Case of $SL_{kn}$}
To describe $^{SL_{kn}} \cY_{k\varpi_n}^{\tau(\lambda)}$, it will be convenient to write elements of $SL_{kn}$ as block matrices
$$ \left( \begin{array}{c|c} a & b \\ \hline c & d \end{array} \right) $$
where $a$ is $(kn-n)\times (kn-n)$, $d$ is $n\times n$, etc. With this convention, we have 
$$ t^{w_0 (k \varpi_n)} = \left( \begin{array}{c|c} t^{-1} I & 0 \\ \hline 0 & t^{k-1} I \end{array} \right), $$
and following (\ref{eq: Gmu}) and (\ref{eq: cosets for Grmu}) we find that
\begin{equation}
^{SL_{kn}} \cW_{k\varpi_n} = \left\{ \left( \begin{array}{c|c} a & b \\ \hline c & d \end{array} \right) \in SL_{kn}((t^{-1})) : \begin{array}{l} a \in t^{-1} I + t^{-2} M_{(k-1)n\times (k-1)n}[[t^{-1}]], \\ b\in t^{-2} M_{(k-1)n\times n}[[t^{-1}]], \\ c\in t^{-2} M_{n\times (k-1)n}[[t^{-1}]], \\ d\in t^{k-1} I + t^{k-2} M_{n\times n}[[t^{-1}]] \end{array} \right\}
\end{equation}

We can also describe $\tau(\lambda)$ more explicitly:
\begin{Lemma}
\label{lem:simpcomp}
We have:
\begin{itemize}
\item[(a)] $\tau(\lambda) = \sum_{i=1}^{n-1} m_{n-i} \alpha_i + k \varpi_n $

\item[(b)] $\tau(\lambda) \leq k n \varpi_1$
\end{itemize}
\end{Lemma}
\begin{proof}
\mbox{}
\begin{itemize}
\item[(a)]  Observe that $\tau(\alpha_i) = \alpha_i$ for $1\leq i \leq n-2$, while $\tau(\alpha_{n-1}) = \alpha_{n-1}+\varpi_n$.  Since $k = m_1$ is the coefficient of $\alpha_{n-1}$ in $\lambda$, the claim follows.

\item[(b)] The difference $kn \varpi_1 - \lambda$ is a linear combination of $\alpha_1,\ldots, \alpha_{n-2}$ with non-negative coefficients: $\alpha_{n-1}$ does not appear.  Since $\tau(\alpha_i) = \alpha_i$ for $1\leq i \leq n-2$,
$$ \tau( kn \varpi_1 - \lambda) = kn \varpi_1 - \tau(\lambda)$$ 
is also a linear combination of $\alpha_1,\ldots, \alpha_{n-2}$ with non-negative coefficients. 
\end{itemize}
\end{proof}

To apply Proposition \ref{prop: slice via coset reps} to $^{SL_{kn}} \cY_{k\varpi_n}^{\tau(\lambda)}$, we must compute the pairings $\langle \tau(\lambda), w_0 \varpi_j^\vee \rangle$ for the fundamental weights $\varpi_j^\vee$ of $SL_{kn}$.  We do so by using Lemma \ref{lem:simpcomp}(a), together with the expansion
$$ k \varpi_n = \sum_{j=1}^{n-1} j (k-1) \alpha_j  + \sum_{j=n}^{kn} (kn-j) \alpha_j $$
Altogether, we find that $^{SL_{kn}} \cY_{k\varpi_n}^{\tau(\lambda)} \subset ^{SL_{kn}} \cW_{k\varpi_n}$ is the closed subscheme defined by the conditions
\begin{equation}
\label{eq: kn presentation}
^{SL_{kn}} \cY_{k\varpi_n}^{\tau(\lambda)} = \left\{ \left( \begin{array}{c|c} a & b \\ \hline c & d \end{array} \right) \in ^{SL_{kn}} \cW_{k\varpi_n} :  \begin{array}{l} \ \text{the valuation of any } j \times j \text{ minor is:} \\ \ (a)\ \geq -j ,  \text{ for } 1\leq j \leq kn-n, \\  \begin{array}{l}(b)\  \geq -m_{j-kn+n} - (kn -j)(k-1), \\ \ \ \text{ for } j > kn-n\end{array} \end{array}\right\}
\end{equation}

\subsection{The isomorphism}
We begin with the following observation regarding $^{SL_{kn}} \cY_{k\varpi_n}^{\tau(\lambda)}$:
\begin{Lemma}
The matrix coefficients of $a - t^{-1} I, b$ and $c$ are zero in $\CC[^{SL_{kn}} \cY_{k\varpi_n}^{\tau(\lambda)}]$.
\end{Lemma}
\begin{proof}
The $^{SL_{kn}} \cY_{k\varpi_n}^{\tau(\lambda)}$ conditions tell us in particular that the valuation of any $1\times 1$ minor must be $\geq -1$.  Since $a \in t^{-1} I + t^{-2} M_{(k-1)n \times (k-1)n}[[t^{-1}]]$, this $1\times 1$ condition is only satisfied if $a = t^{-1} I$.  Similarly for $b, c$.
\end{proof}

The next result establishes Proposition \ref{PropCalc}:
\begin{Proposition}
\label{main prop}
There is an isomorphism of schemes 
$$ ^{SL_n} \cY_0^\lambda \ \ \stackrel{\sim}{\longrightarrow} \ \ ^{SL_{kn}} \cY^{\tau(\lambda)}_{k \varpi_n} $$
defined by
$$ g \mapsto \left( \begin{array}{c|c} t^{-1}I & 0 \\ \hline 0 & t^{k-1} g \end{array} \right) $$
\end{Proposition}
\begin{proof}
We will show that the $^{SL_n} \cY_0^\lambda$ conditions on $g$ imply the $^{SL_{kn}}\cY_{k\varpi_n}^{\tau(\lambda)}$ conditions on its image.  The converse is similar.

Consider a $j\times j$ minor of the image of $g$.  To be nonzero, it must correspond to an $i\times i$ minor of $t^{k-1} g$ times a $(j-i)\times(j-i)$ minor of $t^{-1} I$. In other words, if nonzero, its valuation has the form
$$ i(k-1) + \operatorname{val}(\Delta) - (j-i) = \operatorname{val}(\Delta) + ki-j$$
where $\Delta$ is an $i\times i$ minor of $g$. By the $^{SL_n} \cY_0^\lambda$ condition on $\operatorname{val}(\Delta)$, this is greater than or equal to
\begin{equation}
\label{eq: minor 1}
- m_i + ki -j
\end{equation}
We consider the cases $j\leq kn-n$ and $j > kn-n$ separately, as in (\ref{eq: kn presentation}).

If $j\leq kn-n$, then we must show that
$$ -m_i + ki - j \geq -j $$
Recalling that $kn \varpi_1 \geq \lambda$, the above follows from the inequality
$$ki = \langle kn \varpi_1, \varpi_{n-i}^\vee \rangle \geq \langle \lambda, \varpi_{n-i}^\vee\rangle = m_i$$

If $j > kn -n$, then we must show that
$$ -m_i  +ki -j \geq - m_{j-kn+n} - (kn-j)(k-1) $$
or equivalently
$$ m_i - m_{j-kn+n} \leq k \big( i - (j-kn+n)\big) $$
Note that $j-i \leq kn-n$, so $j -kn +n \leq i$.  Since $\lambda$ is dominant, it follows that $m_{\ell +1} - m_\ell \leq m_1$ for all $\ell$ (e.g. if we think of $\lambda = (p_1 \geq \ldots \geq p_n)$ with $\sum p_i = 0$, then $p_{n-\ell} = m_\ell - m_{\ell+1}$ and $p_n = -m_1$).  Since $m_i - m_{j-kn+n}$ is a telescoping sum of terms $m_{\ell+1} - m_\ell$, and since $m_1 = k$, the inequality follows.

In either case, we see that the $^{SL_{kn}} \cY_{k\varpi_n}^{\tau(\lambda)}$ conditions hold on the image of $g$, as claimed.
\end{proof}

%%%%%%%%%%%%%%%%%%%%%%%%%%%%%%%%%%%%%%%%%%%%%%%%%%%%%%%%%%%%%%
%%%%%%%%%%%%%%%%%%%%%%%%%%%%%%%%%%%%%%%%%%%%%%%%%%%%%%%%%%%%%%

\section{Connection to quiver gauge theories}
\subsection{A general isomorphism between slices}
\label{sec:bfn}
We've now shown that $\cY_\mu^\lambda$ is reduced in type A and hence is isomorphic to $\Gr_\mu^\lambda$.  In particular Proposition \ref{PropCalc} now says that we have an isomorphsm
\begin{align}
  \label{eq:19}
  \leftexp{SL_n}{\Gr^\lambda_0} \cong \leftexp{SL_{kn}}{\Gr^{\tau(\lambda)}_{k\varpi_n}}
\end{align}
In this section we show that this isomorphism has a natural interpretation in the context of Coulomb branches of quiver gauge theories, based on their description by Braverman, Finkelberg and Nakajima \cite{BFN2}.

We'll work more generally, for $G$ an arbitrary simply-laced semisimple group of Dynkin type $I$.  Consider $\lambda \geq \mu$ dominant $G$-coweights, and as per usual we denote $\lambda_i = \langle \lambda, \alpha_i^\ast\rangle$, $m_i = \langle \lambda-\mu, \varpi_i^\ast\rangle$.   Consider the vector spaces $ W_i = \CC^{\lambda_i}$ and $V_i = \CC^{m_i}$ for $i\in I$, and the group $\bfnG := \prod_{i\in I } GL(V_i)$. Fix an orientation $\Omega$ of the Dynkin diagram $I$, and define
\begin{align}
  \label{eq:21}
  \bfnN := \bigoplus_{i\rightarrow j\in \Omega} \hom(V_i, V_j) \oplus \bigoplus_{i\in I} \hom(W_i, V_i),
\end{align}
which is naturally a representation of $\bfnG$.

To this data there is an associated commutative algebra $\cA = \cA_{\bfnG, \bfnN}$, which is a graded Poisson algebra, arising as a special case of the general construction \cite[Section 3(iv)]{BFN1} (see also \cite[Section 3(iii)]{BFN2}).  Consider
\begin{align}
  \label{eq:22}
   \cM_C := \Spec  \cA
\end{align}
This is proposed as a mathematical definition of the Coulomb branch associated to a $3d$ $\cN =4 $ quiver gauge theory.

\begin{Theorem}[\mbox{\cite[Theorem 3.10]{BFN2}}]
\label{Theorem: BFN theorem}
For any dominant coweights $\lambda \geq \mu$ there is an isomorphism of Poisson varieties
\begin{align}
	\label{eq:23}
\Gr_{\mu}^{\lambda} \cong \cM_C
\end{align}

\end{Theorem}

Let $\tau: I \hookrightarrow J$ be an inclusion of (simply-laced) Dynkin diagrams. For $j\in J$, consider the vector spaces
\begin{align}
	\label{eq:24}
\widetilde{W}_j = \left\{ \begin{matrix} W_i, & j = \tau(i), \\ 0, & \text{else} \end{matrix} \right., \qquad \widetilde{V}_j = \left\{ \begin{matrix} V_i, & j = \tau(i), \\ 0, & \text{else} \end{matrix} \right.
\end{align}
as well as the associated coweights $ \widetilde{\lambda}, \widetilde{\mu}$. 

Fix an orientation of $J$ extending that of $I$.  Define the group $\widetilde{\bfnG} := \prod_{j\in J} GL(\widetilde{V}_j) $, its representation $\widetilde{\bfnN}$ analogous to $\bfnN$ above, and the corresponding algebra $\widetilde{\cA}$.  In other words, we are extending our data from the quiver of $I$ to that of $J$ by ``padding by zero''.

Clearly, there are compatible isomorphisms $\bfnG \cong \widetilde{\bfnG}$ and $\bfnN \cong \widetilde{\bfnN}$. From the definitions, there is therefore an isomorphism of graded algebras
\begin{align}
	\label{eq:25}
\cA \cong \widetilde{\cA}
\end{align}
In fact this is an isomorphism of graded Poisson algebras, since it lifts to an isomorphism of their deformations $\cA_\hbar \cong \widetilde{\cA}_\hbar$ (these deformations are defined as in \cite[Section 3(vi)]{BFN1}).  By applying Theorem \ref{Theorem: BFN theorem}, we get:

\begin{Proposition}
\label{Proposition: Isomorphism via BFN}
The isomorphism (\ref{eq:25}) induces a Poisson isomorphism
\begin{align}
	\label{eq:26}
\leftexp{I}{\Gr^\lambda_\mu} \cong \leftexp{J}{\Gr^{\widetilde{\lambda}}}_{\widetilde{\mu}}
\end{align}
\end{Proposition}

\begin{Remark}
\label{rmk: comparing with bfn}
Consider the inclusion $\tau$ of $I= \{1,\ldots, n-1\}$ into $J =\{1,\ldots, kn-1\}$, defined by $\tau(i) = kn-n+i$ (this agrees with our previous conventions, up to twisting by the longest elements of the symmetric groups $S_n$ and $S_{kn}$). For the dimension vectors on $I$ corresponding to $\lambda \geq \mu =0$, we recover the isomorphism (\ref{eq:19}): this follows from the decription of Theorem \ref{Theorem: BFN theorem} in terms of generalized minors given in \cite[Appendix B]{BFN2}.  
\end{Remark}

\begin{eg}
As a variation on this construction, consider a slice of the form $\Gr_{\lambda - \alpha_i}^\lambda$.  On the one hand, it was shown in \cite[Example 2.2]{KWWY} that this variety is Poisson isomorphic to the Kleinian singularity $\CC^2 / (\ZZ / n) $, where $n = \langle \lambda, \alpha_{i^\ast}^\vee \rangle $.   On the other hand, the quiver corresponding to $\Gr_{\lambda - \alpha_i}^\lambda$ has $V_{i^\ast} = \CC$ and $V_{j^\ast} = 0$ for $j\neq i$, so we have $\bfnG = GL(1)$ and $ \bfnN = \hom(\CC^n, \CC) $.  In particular, there is an isomorphism to data corresponding to the Dynkin diagram of $SL(2)$, explaining the above ubiquitous appearance of Kleinian singularities.
\end{eg}

\subsection{An isomorphism of truncated shifted Yangians}
\label{sec:truncatediso}
In this section we'll explicitly describe the quantum analog of the isomorphism of Proposition \ref{Proposition: Isomorphism via BFN}, via algebras $^I Y_\mu^\lambda$ defined from Yangians (c.f. Section \ref{intro: truncated shifted Yangians}).  

In fact there is a family $^I Y_\mu^\lambda(\bc)$ of such algebras, where $\bc$ are certain parameters.  We refer the reader to \cite[Section 3]{KTWWY} for the precise definition of $^IY_\mu^\lambda(\bc)$; we'll only recall the parts of the definition we'll need.   The Yangian $^IY$ has a presentation with generators $F_i^{(r)},H_i^{(r)}$, and $E_i^{(r)}$, where $r=1,2...$ and $i\in I$.  The shifted Yangian $^IY_\mu \subset ^IY$ is the subalgebra generated by all $H_i^{(r)}$ and $E_i^{(r)}$, and the $F_i^{(s)}$ such that $s>\langle \mu^*,\alpha_i\rangle$.  Note that $^I Y$ and $^I Y_\mu$ can also be defined as $\CC[\hbar]$--algebras, and we are working with the specialization $\hbar =1$.

Now we fix $\bc=(\bc_i)_{i\in I}$, where each $\bc_i$ is a multiset of complex numbers with $|\bc_i| = \lambda_i$.  From this data we define series $r_i(u)$ (see \cite[Section 3.2]{KTWWY}) and new Cartan generators $A_i^{(s)}$ using the formula
$$
H_i(u)=r_i(u)\frac{\prod_{j\sim i}A_j(u-\frac{1}{2})}{A_i(u)A_i(u-1)},
$$  
where $H_i(u)=1+\sum_{s>0}H_i^{(s)}u^{-s}$ and $A_i(u)=1+\sum_{s>0}A_i^{(s)}u^{-s}$.  Then $^IY_\mu^\lambda(\bc)$ is the quotient of $Y_\mu$ by the two-sided ideal generated by $A_i^{(s)}$ where $i\in I$ and $s>m_i$.  

Let $\tau(\bc)$ be the collection of multisets indexed by $J$, where $\tau(\bc)_{\tau(i)}=\bc_i$ and otherwise $\tau(\bc)_{j}$ is empty if $j \notin \tau(I)$.

\begin{Proposition}
\label{Prop: isomorphism of Yangians}
We have an isomorphism of algebras
\[
^IY_\mu^\lambda(\bc) \cong {^JY_{\widetilde{\mu}}^{\tau(\lambda)}(\tau(\bc))}.
\]
\end{Proposition}

\begin{proof}
Consider the inclusion of algebras $^IY \hookrightarrow ^JY$ defined by $X_i^{(r)} \mapsto X_{\tau(i)}^{(r)}$, for $X=F,H,E$.
For any $\ell \in I$
\begin{align*}
\langle \widetilde{\mu} , \alpha_{\tau(\ell)} \rangle &= \langle \tau(\lambda),\alpha_{\tau(\ell)} \rangle -\sum_{i\in I}m_i  \langle \alpha_{\tau(i)},\alpha_{\tau(\ell)} \rangle \\
&= \lambda_\ell -\sum_{i \in I} m_ia_{i\ell} \\ &= \langle \mu, \alpha_\ell \rangle.
\end{align*}
Therefore for any $F_\ell^{(r)} \in ^IY_\mu$ we have that $F_{\tau(\ell)}^{(r)} \in ^JY_{\widetilde{\mu}}$, and hence $^IY_\mu\hookrightarrow ^JY_{\widetilde{\mu}}$.  

Define $\mathcal{I}_{\widetilde{\mu}}^{\tau(\lambda)}$ to be  the two sided ideal in $^JY_{\widetilde{\mu}}$ generated by 
\begin{align*}
&A_{\tau(i)}^{(s)} \quad s>m_i, i\in I,\\ \quad &A_j^{(r)} \quad r>0, j\in J\setminus \tau(I).
\end{align*}
By construction we have the s.e.s.
$$
0 \longrightarrow \mathcal{I}_{\widetilde{\mu}}^{\tau(\lambda)} \longrightarrow ^JY_{\widetilde{\mu}} \longrightarrow ^JY_{\widetilde{\mu}}^{\tau(\lambda)}(\tau(\bc)) \longrightarrow 0. 
$$
Composing with the inclusion of the previous paragraph we obtain a map
\[
\varphi: ^IY_\mu\longrightarrow ^JY_{\widetilde{\mu}}^{\tau(\lambda)}(\tau(\bc)).
\]

We make some observations about $\varphi$.  First, $\varphi(A_i^{(r)})=A_{\tau(i)}^{(r)}$.  Indeed, in $^JY_{\widetilde{\mu}}^{\tau(\lambda)}(\tau(\bc))$ we have $A_j(u)=1$ for $j\in J\setminus \tau(I)$.  Hence for any $i\in I$ the following equality in $^JY_{\widetilde{\mu}}^{\tau(\lambda)}(\tau(\bc))$ holds:
\[
H_{\tau(i)}(u)=r_{\tau(i)}(u)\frac{\prod_{\ell \sim i}A_{\tau(\ell)}(u-\frac{1}{2})}{A_{\tau(i)}(u)A_{\tau(i)}(u-1)}
\]
Since $r_{\tau(i)}(u)=r_i(u)$ and $\varphi(H_i(u))=H_{\tau(i)}(u)$, this implies that $\varphi(A_i(u))=A_{\tau(i)}(u)$.  

Since $^IY_\mu^\lambda(\bc)$ is the quotient of $^IY_\mu$ by the ideal generated by 
\[
A_{i}^{(s)} \quad s>m_i, i\in I,
\]
$\varphi$ factors through a map $\varphi':{^IY_\mu^\lambda(\bc)} \longrightarrow ^JY_{\widetilde{\mu}}^{\tau(\lambda)}(\tau(\bc))$.

Second, $\varphi$ is surjective.  To see this first note that in $^JY_{\widetilde{\mu}}^{\tau(\lambda)}(\tau(\bc))$, we have $X_j^{(r)}=0$ for any $r>0, j\in J\setminus I$, and $X=F,H,E$.  This follows since $A_j(u)=1$ and the relations
\begin{align*}
(u-v)[A_j(u),E_j(v)]&=A_j(u)(E_j(u)-E_j(v)), \\
(u-v)[A_j(u),F_j(v)]&=(F_j(u)-F_j(v))A_j(u), \\
[E_j^{(r)},F_j^{(s)}]&=H_j^{(r+s-1)}.
\end{align*}
Therefore $^JY_{\widetilde{\mu}}^{\tau(\lambda)}(\tau(\bc))$ is generated by the $X_{\tau(i)}^{(r)}$, $i\in I$, and these generators are all in the image of $\varphi$.

It follows that $\varphi'$ is also surjective.  We can define a map in the opposite direction by declaring $X_{\tau(i)}^{(r)} \mapsto X_i^{(r)}$ for any $i\in I, r>0$.  This defines an algebra homomorphism.  Indeed any relation only involving generators over $\tau(i)$'s gets mapped to the same relation only involving generators over $i$'s.  Moreover, any relation involving generators over $j$'s, where $j\in J\setminus \tau(I)$, is already zero in $^JY_{\widetilde{\mu}}^{\tau(\lambda)}(\tau(\bc))$.  Finally, as we saw above, this identifies $A_{\tau(i)}(u)$ with $A_{i}(u)$.  Since this map is manifestly surjective it defines an inverse to $\varphi'$.
\end{proof}

\begin{Remark}
With notation as Section \ref{sec:bfn}, there is a deformation of the (quantum) Coulomb branch using the ``flavour symmetry'' group $\bfnG_F := \prod_i GL(W_i)$ \cite[Section 3(v)]{BFN2}.  The data $\bc$ defines a specialization $H^\ast_{\bfnG_F\times \CC^\times}(pt)\rightarrow \CC$ (we also specialize $\hbar = 1$), and a corresponding specialized algebra $\cA_{\hbar = 1,\bc}$. There is a surjection $^I Y_\mu^\lambda(\bc) \twoheadrightarrow \cA_{\hbar = 1, \bc}$ by \cite[Appendix B]{BFN2}. Appropriately twisting data by factors of $w_0$ as in Remark \ref{rmk: comparing with bfn}, one can verify that this surjection intertwines the isomorphism from Proposition \ref{Prop: isomorphism of Yangians} with an isomorphism  $\cA_{\hbar=1,\bc} \cong \widetilde{\cA}_{\hbar=1, \tau(\bc)}$ analogous to Proposition \ref{Proposition: Isomorphism via BFN}
\end{Remark}

%%%%%%%%%%%%%%%%%%%%%%%%%%%%%%%%%%%%%%%%%%%%%%%%%%%%%%%%%%%%%%
%%%%%%%%%%%%%%%%%%%%%%%%%%%%%%%%%%%%%%%%%%%%%%%%%%%%%%%%%%%%%%

\end{document}